\documentclass[11pt, oneside]{amsart}   	
\usepackage{geometry}  
\usepackage{graphicx}
\usepackage{array}
\usepackage{amssymb}
\usepackage{amsmath}
\usepackage{amsthm}
\usepackage{graphicx}
\usepackage[parfill]{parskip} 
\usepackage[utf8]{inputenc}
\usepackage[english]{babel}
\usepackage{algorithm}
\usepackage{tikz}
\usepackage[noend]{algpseudocode}
\usepackage{caption}
\usepackage[colorlinks=true,linkcolor=blue]{hyperref}
\usepackage{subcaption}
\usepackage{fancyhdr}
\makeatletter
\def\BState{\State\hskip-\ALG@thistlm}
\makeatother

\theoremstyle{definition}
\newtheorem{thm}{Theorem}[section]
\theoremstyle{definition}
\newtheorem{cor}[thm]{Corollary}
\theoremstyle{definition}
\newtheorem{prop}[thm]{Proposition}
\theoremstyle{definition}
\newtheorem{dfn}[thm]{Definition}
\theoremstyle{definition}
\newtheorem{lem}[thm]{Lemma}
\theoremstyle{definition}
\newtheorem{ex}[thm]{Example}
\theoremstyle{definition}
\newtheorem{conj}[thm]{Conjecture}

\graphicspath{ {images/} }

\begin{document}

\title{Classification of Minimal Polygons with Specified Singularity Content}
\author{Daniel Cavey and Edwin Kutas}

\maketitle

\begin{abstract}
It is known that there are only finitely many mutation-equivalence classes with a given singularity content, and each of these equivalence classes contains only finitely many minimal polygons. We describe an efficient algorithm to classify these minimal polygons. To illustrate this algorithm we compute all mutation-equivalence classes of Fano polygons with basket of singularities given by (i) $\mathcal{B} = \{ m_{1} \times \frac{1}{3}(1,1), m_{2} \times \frac{1}{6} (1,1) \}$ and (ii) $\mathcal{B} = \{ m \times \frac{1}{5}(1,1) \}$.
\end{abstract}

\section{Introduction.}

Let $N$ be a lattice. A \emph{polytope} $P$ in $N_{\mathbb{R}} = N \otimes \mathbb{R}$ is a set of the form 
\[ P = \Big\{ \sum\limits_{u \in S} \lambda_{u} u : \lambda_{u} > 0 \text{ and } \sum\limits_{u \in S} \lambda_{u} = 1 \Big\}, \]
where $S \subset N_{\mathbb{R}}$ is a finite set of points. A \emph{Fano polytope} is a full-dimensional convex polytope such that the vertices $\mathcal{V}(P) \in N$ are all primitive, and that the origin lies in the strict interior of $P$. When $N$ is a rank-two lattice $P$ is known as a \emph{Fano polygon}. We always consider polytopes up to $GL(N)$-equivalence.

The span of each face $E$ of a Fano polygon $P$, by which we mean $\mathbb{R}_{\geq 0}E$, defines a cone. Alternatively this is the polyhedral cone whose primitive generating vertices are given by the endpoints of $E$. We obtain a fan in $N_{\mathbb{R}}$ corresponding to $P$, and this determines a toric del Pezzo surface $X_{P}$. Many properties of $X_{P}$ have combinatorial analogues in the Fano polygon $P$; examples include the singularities and the anticanonical degree $(-K_{X_{P}})^{2}$. Toric geometry can be further studied in \cite{ToricVarieties, IntroductiontoToricVarieites}.

The dual lattice of $N$ is $M = \text{Hom}(N,\mathbb{Z})$ with the pairing $\langle , \rangle: N \times M \rightarrow \mathbb{Z}$. The \emph{lattice length} of a line segment $E \subset N_{\mathbb{R}}$ is given by the value $\lvert E \cap N \rvert - 1$. The \emph{lattice height} of a line segment is given by the lattice distance from the origin: that is, given a primitive inner pointing normal $\omega_{E}$ of $E$ belonging to $M$, the height is given by $\lvert \langle v,n_{E} \rangle \rvert$, for any $v \in E$.

The motivation for the paper comes from an attempt to classify Fano varieites using mirror symmetry. An understanding of mirror symmetry can be gained from Coates--Corti--Galkin--Golyshev--Kasprzyk \cite{MirrorSymmetryandFanoManifolds}. Mirror symmetry suggests that classifying Fano polytopes could help to classify Fano varieties. To a Fano polytope $P$, we have an associated toric variety $X_{P}$.

More specifically we study Fano polytopes up to mutation-equivalence. A mutation is a combinatorial operation on $P \subset N_{\mathbb{R}}$ introduced by Akhtar--Coates--Galkin--Kasprzyk \cite{MinkowskiPolynomialsandMutations}, which will be described in section \ref{Section 2}. We also use a mutation invariant of Fano polygons, introduced in \cite{SingularityContent} by Akhtar--Kasprzyk, known as the singularity content. Equivalently we consider Fano varieties up to qG-deformation, see Akhtar--Coates--Corti--Heuberger--Kasprzyk--Oneto--Petracci--Prince--Tveiten and Koll\'{a}r--Shepherd-Barron \cite{MirrorSymmetryandtheClassificationofOrbifoldDelPezzoSurfaces, ThreefoldsandDeformationsofSurfaceSingularities}. 
\bigskip

\begin{dfn}[\cite{MirrorSymmetryandtheClassificationofOrbifoldDelPezzoSurfaces}]
A del Pezzo surface with cyclic quotient singularities that admits a qG-deformation to a normal toric del Pezzo surface is said to be of \emph{class TG}.
\end{dfn}

The reason we consider Fano polytopes and Fano varieties up to their respective equivalence classes lies in the following conjecture:
\bigskip

\begin{conj}\cite[Conjecture A]{MirrorSymmetryandtheClassificationofOrbifoldDelPezzoSurfaces}
There exists a bijective correspondence between the set of mutation-equivalence classes of Fano polygons and the set of qG-deformation equivalence classes of locally qG-rigid TG del Pezzo surfaces with cyclic quotient singularities.
\end{conj}

Recent results from Corti--Heuberger and Kasprzyk--Nill--Prince \cite{DelPezzoSurfaceswith1/311points, MinimalityandMutationEquivalenceofPolygons} support this conjecture.
\bigskip

\begin{thm}\cite[Theorem 1.2]{MinimalityandMutationEquivalenceofPolygons}
There are precisely ten mutation-equivalence classes of Fano polygons such that the toric del Pezzo surface $X_{P}$ has only T-singularities. They are in bijective correspondence with the ten families of smooth del Pezzo surfaces.
\end{thm}
\bigskip

\begin{thm}\cite{DelPezzoSurfaceswith1/311points, MinimalityandMutationEquivalenceofPolygons}
There are precisely 29 qG-deformation families of del Pezzo surfaces with $m\geq 1$ singular points of type $\frac{1}{3}(1,1)$ and precisely 26 of these are of class TG. They are in bijective correspondence with 26 mutation-equivalence classes of Fano polygons with singularity content $(n,\{ m \times \frac{1}{3}(1,1) \}), m \geq 1$.
\end{thm}
\bigskip

Within this context we wish to classify Fano polygons with a given singularity content. Assuming Conjecture A holds, this is equivalent to a classification of locally qG-rigid del Pezzo surfaces admitting a TG degeneration. We hope to provide further evidence that this is indeed the case. The main results of this paper are an efficient algorithm to produce representations (called minimal polygons) for the mutation-equivalent classes with a given singularity content and;
\bigskip

\begin{thm} \label{1.5}
There are precisely 14 mutation-equivalence classes of Fano polygons with singularity content $\Big(n, \{ m_{1} \times \frac{1}{3}(1,1), m_{2} \times \frac{1}{6}(1,1) \} \Big)$ with $m_{1} \geq 0, m_{2} > 0$.
\end{thm}
\bigskip

\begin{thm} \label{1.6}
There are precisely 12 mutation-equivalence classes of Fano polygons with singularity content $\Big(n, \{ m \times \frac{1}{5}(1,1) \} \Big)$ with $m > 0$.
\end{thm}

\subsection*{Overview of Paper:}

\indent In sections \ref{Section 2} and \ref{Section 3}, we cover technical material. This has been inspired mostly from papers \cite{MirrorSymmetryandtheClassificationofOrbifoldDelPezzoSurfaces, MinkowskiPolynomialsandMutations, SingularityContent, MinimalityandMutationEquivalenceofPolygons}. Section \ref{Section 2} introduces the theory of mutations and singularity content. We will give a formal definition of these notions, as well as some useful results and examples. In section \ref{Section 3} we introduce the definition of minimal polygons and provide a number of equivalent conditions. 

Sections \ref{Section 4}--\ref{Section 6}, consist of classification results. In section \ref{Section 4} we describe an algorithm to find a classification of minimal Fano polygons considered up to mutation-equivalence with a fixed singularity content. In sections \ref{Section 5} and \ref{Section 6} we prove Theorems \ref{1.5} and \ref{1.6} respectively, via use of this algorithm. If the resulting Fano polygon can be considered (via mutation if necessary) as a triangle, then we know that the corresponding toric variety will be the quotient of a weighted projective space.

\section{Mutations of Fano Polygons and Singularity Content.} \label{Section 2}

\subsection{Mutations}

We recall the definition of the Minkowski sum of lattice polygons.
\\

\begin{dfn}
Let $P,Q \subset N_{\mathbb{R}}$ be two lattice polytopes. We define the \emph{Minkowski sum} of $P$ and $Q$ by
\[ P+Q = \{ p+q : p \in P, q \in Q \}. \]
By convention $P + \varnothing = \varnothing$.
\end{dfn}

Let $P \subset N_{\mathbb{R}}$ be a polygon, and $E$ be an edge of $P$. Consider the primitive inward pointing normal $\omega_{E} \in M$ of this edge. This vector can be thought of as a grading function on the polygon $P$. For $h \in \mathbb{Z}$, define
 \[ \omega_{h}(P) = \text{conv} \{ v \in N \cap P : \omega_{E}(v) = h \}. \]
Note that $\omega_{h}(P)$ may be empty (indeed it will be for infinitely many values of $h$) and that $\omega_{E}(E) = -r_{E}$, where $r_{E}$ is the height of E. Choose $v_{E}$ to be a primitive vector of the lattice $N$ such that $\omega_{E}(v_{E}) =0$. Note this is uniquely defined up to sign. Set $F = \text{conv}\{\mathbf{0},v_{E}\}$. $F$ is a line of lattice length $1$ and height $0$, that is parallel to $E$.
\\

\begin{dfn}
For all $h < 0$, suppose that there exists $G_{h} \subset N_{\mathbb{R}}$ such that
\[ \{ v \in \mathcal{V}(P) : \omega_{E}(v) = h \} \subseteq  G_{h} + \lvert h \rvert F   \subseteq \omega_{h}(P) . \]
In the case $\omega_{h}(P) = \varnothing$ the inclusion holds taking $G_{h} = \varnothing$. Then we define the \emph{mutation} of P given by $\omega_{E}$, $F$ and $G_{h}$ to be
\[ \text{mut}_{(\omega_{E},F)}(P) = \text{conv} \Big( \underset{h<0}{\bigcup} G_{h} \cup \underset{h > 0}{\bigcup} (\omega_{h}(P) + hF)  \Big) \subset N_{\mathbb{R}} \]
\end{dfn}

$\text{mut}_{(\omega,F)}(P)$ is independent of the choice of $G_{h}$ (assuming existence of the mutation). Hence our notation for a mutation makes no reference to $G_{h}$. If there is no possible choice of $G_{h}$, then we cannot define a mutation. We understand exactly when this is the case:

\begin{lem}\cite[Lemma 2.3]{MinimalityandMutationEquivalenceofPolygons}
Let $E$  be an edge of a Fano polygon $P$ with primitive inner normal vector $\omega_{E} \in M$. Then $P$ admits a mutation with respect to $\omega$ if and only if $\lvert E \cap N \rvert \geq r_{E}$.
\end{lem}
\bigskip

\begin{ex} \label{2.4}
Consider the polygon $P = \text{conv} \{ (1,0), (0,1), (-5,-1) \}$ corresponding to the weighted projective space $\mathbb{P}(1,1,5)$. This will be used as running example throughout the paper. P has three edges, however by Lemma 2.3 only two of these edges will determine a mutation. We mutate with respect to the edge $E = \text{conv}\big\{(1,0),(0,1)\big\}$. The primitive inner pointing normal is given by $\omega_{E} = (-1,-1) \in M$. Set $F = \text{conv} \big\{ \mathbf{0},(1,-1)\big\}$, $G_{-1}=\{(0,1)\}$ and $G_{h} = \varnothing$ for $h<-1$. We calculate the mutation of $P$ with respect to the primitive inner point normal $\omega_{E}$, the factor $F$ and the polygon $G_{-1}$. Then:

\begin{align*}
Q &= \text{mut}_{(\omega,F)}(P) \\ 
&= \text{conv} \Big\{ \big(G_{-1}\big) \cup \big(\omega_{0}(P)\big) \cup \big(\omega_{1}(P) + F \big) \cup \big(\omega_{2}(P) + 2F \big) \cup \cdots \cup \big(\omega_{7}(P) + 7F \big) \Big\} \\ 
&= \text{conv} \Big\{ (0,1) , (-5,-1), (1,-7) \Big\}.
\end{align*}

Q corresponds to the toric variety $\mathbb{P}(1,5,36)$.
\end{ex}

We have a number of additional properties of mutations:
\begin{itemize}
\item Since we always consider polygons up to $GL(N)$-equivalence, we have that $\text{mut}_{(\omega,F)}(P) = \text{mut}_{(\omega,-F)}(P)$. Hence the sign of $v_{E}$ is  not important.
\item Mutation is invertible: If $Q = \text{mut}_{(\omega,F)}(P)$, then $P = \text{mut}_{(-\omega,F)}(Q)$.
\item \cite[Proposition 2]{MinkowskiPolynomialsandMutations} $P$ is a Fano polytope if and only if $\text{mut}_{(\omega,F)}(P)$ is a Fano polytope.
\end{itemize}
\bigskip

\begin{dfn}
Let $P,Q \subset N_{\mathbb{Q}}$ be two Fano polygons. We say $P$ and $Q$ are \emph{mutation-equivalent} if there exists a finite sequence of polygons $P_{0},P_{1},\cdots, P_{n}$ such that $P_{0} \cong P$, $P_{n} \cong Q$ and, for all $i \in \{ 0,\cdots, n-1 \}$, we have that $P_{i+1} = \text{mut}_{(\omega_{i},F_{i})}(P_{i})$ for some $\omega_{i}$ and $F_{i}$.
\end{dfn}

Mutation-equivalence defines an equivalence relation.

\subsection{Singularity Content}

We recall the definition of singularity content introduced in \cite{SingularityContent}.

Recall that a \emph{quotient singularity} $\frac{1}{R}(a,b)$ is given by the action of $\mu_{R}$ on $\mathbb{C}^{2}$ by $(x,y) \mapsto (\epsilon^{a}x,\epsilon^{b}y)$ where $\epsilon$ is an $R^{\text{th}}$ root of unity, and considering $Z = \text{Spec}(\mathbb{C}[x,y]^{\mu_{R}})$. The germ of the origin is the singularity.

A \emph{cyclic quotient singularity} is a quotient singularity $\frac{1}{R}(a,b)$ such that $\gcd(R,a)=1$ and $\gcd(R,b)=1$. In this situation, set $k = \gcd(a+b,R)$ so that $a+b=kc$ and $R=kr$. We can write the cyclic quotient singularity as $\frac{1}{kr}(1,kc-1)$.
\bigskip

\begin{dfn}
A cyclic quotient singularity $\frac{1}{kr}(1,kc-1)$ is a \emph{T-singularity} if $r \mid k$.
\end{dfn}

Every T-singularity can be written in the form $\frac{1}{nr^{2}}(1,nrc-1)$. where $k=nr$ If $n=1$, we refer to the singularity as a primitive T-singularity. Kollar--Shepherd-Barron, \cite{ThreefoldsandDeformationsofSurfaceSingularities}, show a cyclic quotient singularity is a T-singularity if and only if it admits a qG-smoothing. 
\bigskip

\begin{dfn}
A cyclic quotient singularity $\frac{1}{kr}(1,kc-1)$ is an \emph{R-singularity} if $k<r$.
\end{dfn}

A singularity is an R-singularity if and only if it is rigid under qG-deformation.

Consider an arbitrary cyclic quotient singularity $\frac{1}{kr}(1,kc-1)$. By the Euclidean Algorithm there exists unique non-negative integers $n$ and $k_{0}$ such that $k=nr+k_{0}$. If $k_{0}=0$ then the singularity is qG-smoothable. If $k_{0} >0$, then the singularity is qG-deformation equivalent to a $\frac{1}{k_{0}r}(1,k_{0}c-1)$ cyclic quotient singularity. 
\bigskip

\begin{dfn}[\cite{SingularityContent}]
Let $\sigma = \frac{1}{kr}(1,kc-1)$ be a cyclic quotient singularity. By the Euclidean algorithm we have $k=nr+k_{0}$. The \emph{residue} of $\sigma$ is given by 
\[ \text{res}(\sigma) = \begin{cases} \varnothing \text{, if } k_{0} = 0 \\ \frac{1}{k_{0}r}(1,k_{0}c-1) \text{, otherwise.} \end{cases}  \]
The \emph{singularity content} of $\sigma$ is given by the pair 
\[ \text{SC}(\sigma) = \big(n, \text{res}(\sigma)\big). \]
\end{dfn}

This discussion of T-singularities and R-singularities has a natural description in the language of cones. Let $C \subset N_{\mathbb{Q}}$ be a cone with generating rays described by the primitive lattice points $p_{0}$ and $p_{1}$. Consider $E = \text{conv}\{p_{0},p_{1}\}$. Let $h$ be the height of $E$ and $l$ be the lattice length.  By the Euclidean algorithm we have $l=hn+r$. We divide $C$ into separate sub-cones $C_{0},\cdots,C_{n}$ where $C_{1},\cdots,C_{n}$ (known as T-cones) have lattice length $h$, and $C_{0}$ has lattice length $r$ and is known as an R-cone. Each sub-cone corresponds to a singularity of the corresponding toric variety. The T-cones correspond to the T-singularities and the R-cones to the R-singularities. This decomposition is reflected in the cone $C$ corresponding to the cyclic quotient singularity $\frac{1}{kr}(1,kc-1)$, which can be qG-smoothed to the cyclic quotient singularity $\frac{1}{k_{0}r}(1,k_{0}c-1)$ corresponding to the sub-cone $C_{0}$. We define the singularity content of $E$ to be the singularity content of the corresponding cyclic quotient singularity. The singularities are independent of the choice of decomposition.
\bigskip

\begin{dfn}
Let $P \subset N_{R}$ be a polygon. Label the edges of $P$ in clockwise order $E_{1},\cdots E_{k}$. Each edge $E_{i}$ corresponds to a cyclic quotient singularity $\sigma_{i}$ corresponding to this cone. Let $\text{SC}(E_{i}) = \big( n_{i},\text{res}(\sigma_{i}) \big)$. Set $n = \sum\limits_{i=1}^{k} n_{i}$ and $\mathcal{B} = \{ \text{res}(\sigma_{1}) , \cdots, \text{res}(\sigma_{k}) \}$, where $\mathcal{B}$ is an ordered set known as the basket  of R-singularities. We then define the \emph{singularity content} of $P$ to be
\[ \text{SC}(P) = \big( n,\mathcal{B} \big). \]
\end{dfn}

For

\begin{ex}
For $P=\text{conv}\{(0,1),(1,0),(-5,-1)\}$, we calculate that $\text{SC}(P) = \Big( 2, \big\{ \frac{1}{5}(1,1) \big\} \Big)$.
\end{ex}
\bigskip

\begin{prop} \cite[Proposition 3.6]{SingularityContent}
Singularity content is an invariant of Fano polygons under mutation.
\end{prop}

Indeed note that in Example \ref{2.4}, that $\text{SC}(P) = \text{SC}(Q) = \big(2,\{\frac{1}{5}(1,1)\}\big)$.

\subsection{Hirzebruch--Jung Continued fractions and Applications to Algebraic Geometry}

There is information about the del Pezzo surface $X_{P}$ corresponding to a polygon $P$ written into the singularity content; $X_{P}$ is qG-deformation equivalent to a del Pezzo surface $X$ such that the topological Euler number $\chi \big(X \backslash \text{Sing}(X)\big) = n$  and the singular points are given by $\text{Sing}(X) = \mathcal{B}$. The anticanonical degree and Hilbert series of $X_{P}$ is totally determined by the singularity content. See \cite{MirrorSymmetryandtheClassificationofOrbifoldDelPezzoSurfaces, SingularityContent}. 
\bigskip

\begin{dfn}
Let $p,q$ be positive coprime integers. Then the \emph{Hirzebruch--Jung continued fraction} is given by the continued fraction of the form:
\[ \frac{p}{q} = a_{1} - \frac{1}{a_{2} - \frac{1}{a_{3} - \frac{1}{\ddots}} } =  [a_{1},\cdots a_{k}]. \]
\end{dfn}

Given a cyclic quotient singularity $\sigma = \frac{1}{R}(1,a-1)$, construct the variety $Z = \text{Spec}(\mathbb{C}[x,y]^{\mu_{R}})$ as in the definition of quotient singularity. We can calculate information about a minimal resolution of $Z$ from the Hirzebruch--Jung continued fraction of $\frac{R}{a-1}$. Consider the minimal resolution $\pi: Y \rightarrow Z$ with
\[ K_{Y} = \pi^{*} K_{Z} + \sum\limits_{i=1}^{k} d_{i}E_{i}. \] 

Let $[a_{1},\cdots,a_{k_{\sigma}}]$ be the Hirzebruch--Jung continued fraction of $\frac{R}{a-1}$. The values $-a_{i}$ are the self-intersection numbers of the exceptional divisors $E_{i}$. Additionally set:
\[ \alpha_{1} = \beta_{k_{\sigma}} = 1, \]
\[ \frac{\alpha_{i}}{\alpha_{i-1}} = [a_{i-1},\cdots,a_{1}], \text{ for } i \in \{ 2,\cdots,k_{\sigma} \}, \]
\[ \frac{\beta_{i}}{\beta_{i+1}} = [a_{i+1},\cdots,a_{k_{\sigma}}], \text{ for } i \in \{ 1,\cdots,k_{\sigma}-1 \}, \]

then the discrepancies are given by $d_{i} = -1 + \frac{\alpha_{i}+ \beta_{i}}{R}$. For further reading on minimal resolutions see Reid \cite{YoungPersonsguidetoCanonicalSingularities}.
\bigskip

\begin{prop}\cite[Proposition 3.3, Corollary 3.5]{SingularityContent} \label{2.13}
Let $X_{P}$ be a complete toric surface with corresponding Fano polygon $P$. Suppose $X_{P}$ has singularity content $(n,\mathcal{B})$. Then
\[ (-K_{X_{P}})^{2} = 12 - n - \sum\limits_{\sigma \in \mathcal{B}} A_{\sigma} , \]
where $A_{\sigma} = k_{\sigma} + 1 - \sum\limits_{i=1}^{k_{\sigma}} d_{i}^{2} a_{i} + 2 \sum\limits_{i=1}^{k_{\sigma}-1} d_{i}d_{i+1}$. Furthermore the anticanonical Hilbert series of $X_{P}$ admits a decomposition
\[ \text{Hilb}(X_{P}, -K_{X_{P}}) = \frac{1+(K_{X_{P}}^{2}-2)t + t^{2}}{(1-t)^{3}} + \sum\limits_{\sigma \in \mathcal{B}} Q_{\sigma}(t), \]
where $Q_{\frac{1}{R}(1,a-1)} = \frac{1}{1-t^{R}} \sum\limits_{i=1}^{R-1}(\delta_{ai}-\delta_{0})t^{i-1}$ is the Riemann-Roch contribution coming from the singularity $\frac{1}{R}(1,a-1)$ and $\delta_{j} = \frac{1}{R}\sum\limits_{\epsilon \in \mu_{R}, \epsilon \neq 1} \frac{\epsilon^{j}}{(1-\epsilon)(1-\epsilon^{a-1})}$ are the Dedekind sums.
\end{prop}
\bigskip

\begin{ex} \label{2.14}
$P=\text{conv}\{(0,1),(1,0),(-5,-1)\}$ has singularity content $\big(2,\{\frac{1}{5}(1,1)\}\big)$. The Hirzebruch--Jung continued fraction of the cyclic quotient singularity $\frac{1}{5}(1,1)$ is simply $[5]$, so $d_{1}  =-\frac{3}{5}$ and $A_{\frac{1}{5}(1,1)} = \frac{1}{5}$. Also $Q_{\frac{1}{5}(1,1)} = \frac{t-2t^{2}+t^{3}}{5(1-t^{5})}$ . Therefore the anticanonical degree and Hilbert series of $X = \mathbb{P}(1,1,5)$ are given by
\[ (-K_{X})^{2} = \frac{49}{5}, \]
\[ \text{Hilb} \big( X, -K_{X} \big) = \frac{1 + 8t + 2t^{3}-2t^{4} -8t^{6}-t^{7}}{(1-t^{5})(1-t)^{3}}. \]
\end{ex}

More generally for a polygon $P$ with $n$ primitive T-singularities and a basket of singularities $\mathcal{B} = \{ m \times \frac{1}{5}(1,1) \}$:
\[ (-K_{X_{P}})^{2} = 12 - n - \frac{1}{5}m, \]
\[ \text{Hilb}(X_{P}, -K_{X_{P}}) = \frac{-t^{7} + (n-10)t^{6} + (m-1)t^{5} -2mt^{4} + 2mt^{3} + (1-m)t^{2} + (10-n)t + 1 }{(1-t)^{3}(1-t^{5})}. \]

Hirzebruch--Jung fractions can be further studied in \cite{IntroductiontoToricVarieites, SurfaceCyclicQuotientSingularitiesandHirzeburchJungresolutions}.

\section{Minimal Fano Polygons.} \label{Section 3}

Mutation-equivalence classes raises the issue about our choice of representative when considering a mutation-equivalence class of polygons. This leads to the definition of a minimal polygon from \cite{MinimalityandMutationEquivalenceofPolygons}. For a polygon $P$, we use the notation $\partial P$ for the boundary of $P$, and $P^{\circ}$ for the interior.
\bigskip

\begin{dfn}
Let $P \subset N_{\mathbb{Q}}$ be a Fano polygon. We say that the polygon $P$ is \emph{minimal} if $\lvert \partial P \cap N \rvert \leq \lvert \partial Q \cap N \rvert$, for all $Q = \text{mut}_{(\omega,F)}(P)$.
\end{dfn}

In \cite{MinimalityandMutationEquivalenceofPolygons}, we are provided with a number of conditions that are equivalent to minimality:
\begin{itemize}
\item $\lvert P^{\circ} \cap N \rvert \leq  \lvert Q^{\circ}\cap N \rvert$, for all $Q = \text{mut}_{(\omega,F)}(P)$ where $P^{\circ}$ denotes the interior of $P$.
\item $\text{Vol}(P) \leq \text{Vol}(Q)$, for all $Q = \text{mut}_{(\omega,F)}(P)$.
\item $r_{1}+\cdots +r_{n} \leq s_{1}+\cdots +s_{n}$, for all $Q = \text{mut}_{(\omega,F)}(P)$, where $r_{i}$ are the Gorenstein indices of the primitive T-cones associated with $P$, and $s_{i}$ are the Gorenstein indices of the primitive T-cones associated with $Q$.
\item For an edge $E$ of $P$, let $\omega_{E} \in M$ be the primitive inner pointing normal of $E$. We define $h_{\text{min}} =  \min \{ \omega_{E}(v) : v \in P \}$ and $h_{\text{max}} =  \max \{ \omega_{E}(v) : v \in P \}$. If $\lvert E \cap N \rvert -1 \geq \lvert h_{\min} \rvert$, then $\lvert h_{\min} \rvert \leq h_{\max}$.
\end{itemize}

In every mutation-equivalence class we can find at least one minimal representative using algorithm 1 below.

\begin{algorithm}
\caption{Calculation of Minimal Polygons}\label{euclid}
\begin{algorithmic}[1]
\State \textbf{Input:} \text{Polygon $P$}
\State $\text{BoundaryPoints} = \lvert \partial P \cap N \rvert$
\For {$Q = \text{mut}_{(\omega,F)}(P)$,}
\If{BoundaryPoints $> \lvert \partial Q \cap N \rvert$}
\State  $P \gets Q$
\State $\text{BoundaryPoints} \gets \lvert \partial Q \cap N \rvert$
\State \textbf{go to} 3.
\EndIf
\EndFor
\State \textbf{Output:} $P$
\end{algorithmic}
\end{algorithm}

This algorithm will always terminate as the number of lattice boundary points of a polygon is finite and non-negative. The minimal representative of a mutation-equivalence class is not necessarily unique. In the search for Fano polygons we always look for minimal representatives of each mutation-equivalence class.
\bigskip

\begin{ex}
In Example \ref{2.4}, it is routine to check that $P$ is minimal.
\end{ex}

\section{Algorithm to calculate Minimal Polygons with fixed basket of singularities $\mathcal{B}$.} \label{Section 4}

\subsection{Special Facets}

We require the notion of a special facet introduced by {\O}bro \cite{AnAlgorithmfortheClassificationofSmoothFanoPolytopes}.
\bigskip

\begin{dfn}
Let $P \subset N_{\mathbb{R}}$ be a Fano polygon. We say that an edge $E$ of $P$ is a \emph{special facet} if
\[ \sum\limits_{\text{vertices } v \in P} v \in \mathbb{R}_{\geq 0}E, \]
\end{dfn}

Every Fano polygon has at least one special facet since $\mathbf{0} \in P^{\circ}$. Crucially we use a result from \cite{OntheCombinatorialClassificationofToriclogdelPezzoSurfaces} which is derived from a proof in \cite{ABoundednessResultforToricLogdelPezzosurfaces}.
\bigskip

\begin{lem} \label{4.2}
Let $P$ be a Fano polygon. Let $F$ be a special facet of $P$ of height $h$ and with inward pointing normal $n_{F} \in M$. Then 
\[ P \subset \{ (a,b) \in N_{\mathbb{R}} : -h(h+1) \leq b \leq h \}. \]
\end{lem}

Given any Fano polygon $P$, we can perform a $GL(N)$-transformation to orientate $P$ such that a special facet $F$ is horizontal, of height $h$ and minimal with respect to the linear function $(0,-1)\in M$ on $P$. By Lemma \ref{4.2} $P$ is bounded below by the line $L= \{v \in N_{\mathbb{Q}} : \langle (0,-1) , v \rangle = h(h+1) \}$.
\bigskip

\begin{ex}
Considering $P = \text{conv}\Big\{ (0,1),(1,0),(-5,-1) \Big\}$, we calculate that :
\[ \sum\limits_{\text{vertices } v \in P} v  = (0,1) + (1,0) + (-5, -1) = (-4,0). \]
So $P$ has a unique special facet given by $F=\text{conv}\Big\{ (0,1),(-5,-1) \Big\}$.
\end{ex}

\subsection{Description of Algorithm}

Define the maximal local index of a Fano polygon $P$ by 
\[ m_{P} = \max \{ h_{E}: E \in \mathcal{F}(P)\}. \]
Similarly define $m_{\mathcal{B}}$ to be the maximum height amongst the cones representing the R-singularities of $P$.

The classification of Fano polygons with a given basket of singularities $\mathcal{B}$ up to mutation-equivalence is split into two cases:

\begin{itemize}
\item Case (i): $m_{P}=m_{B}$
\item Case (ii): $m_{P}>m_{B}$
\end{itemize}

The proof of Theorem 6.3 in \cite{MinimalityandMutationEquivalenceofPolygons} tackles case (ii). Note the polytopes this algorithm outputs are not necessarily minimal. The main result of this paper is an efficient algorithm to deal with case (i). An algorithm to compute the classification for case (i) has been completed in \cite{OntheCombinatorialClassificationofToriclogdelPezzoSurfaces}. However tackling classifications beyond the case of polygons with only $\frac{1}{3}(1,1)$ R-singularities is inefficient. The basic idea is to start with only a single edge, that will eventually be a special facet, and to inductively construct minimal polygons by adding in edges that will contribute T-singularities or R-singularities contained in $\mathcal{B}$. The full algorithm is described below.

\begin{algorithm}
\caption{Classification of Minimal Fano Polygons with given basket of singularities with $m_{P}=m_{\mathcal{B}}$}\label{euclid}
\begin{algorithmic}[1]
\State \textbf{Input:} Special facet $F = \text{conv}\big( (a,l_{F}), (b,l_{F}) \big)$, Basket of singularities $\mathcal{B}$
\State $L1 = \overline{(a,l_{F}), (-\mathcal{B},0)}$
\State $L2 = \overline{(b,l_{F}), (\mathcal{B},0)}$
\State $L = \{ v \in N: \langle v , n_{F} \rangle = l_{F}(l_{F}+1) \}$
\State $T =$ region bounded by $F,L,L1$ and $L2$
\State PossiblePoints $= \{ \text{primitive points } v \in N \text{ contained in } T \}$
\State ActiveConstructions $= \{F\}$
\State CompleteConstructions $= \varnothing$ 
\For {$P \in$ ActiveConstructions}
\For {$v \in$ PossiblePoints}
\If {$v \neq (a,l_{F})$ and the edge $E$ made from adding $v$ to $\mathcal{V}(P)$ satisfies convexity and $m_{E} \leq m_{\mathcal{B}}$}
\State  $\text{ActiveConstructions} \gets \big(\text{ActiveConstructions} \backslash \{P\}\big) \cup \{P \cup E\}$
\EndIf
\If {$v =(a,l_{F})$ and the edge $E$ made from adding $v$ to $\mathcal{V}(P)$ satisfies convexity and $m_{E} \leq m_{\mathcal{B}}$}
\State ActiveConstructions $\gets \text{ActiveConstructions} \backslash \{P\}$
\State CompleteConstructions $\gets \text{CompleteConstructions} \cup \{P  \cup E\}$
\EndIf
\EndFor
\EndFor
\If {ActiveConstructions $\neq \varnothing$}
\State \textbf{go to} 9
\EndIf
\For {$P \in$ CompleteConstructions}
\If {$P$ not minimal or $F$ not a special facet of $P$ or $\{$R-singularities of $P\}\neq \mathcal{B}$}
\State CompleteConstructions $\gets$	CompleteConstructions $\backslash \{P\}$
\EndIf
\EndFor
\State \textbf{Output:} CompleteConstructions
\end{algorithmic}
\end{algorithm} 

There are a few further points to clarify in order to complete a proof showing we obtain a complete classification from the algorithm for a given $\mathcal{B}$. Firstly we prove that there are only finitely many choices of special facet as an input. We run the algorithm for all such possible choices. We check our output up to $GL(N)$-equivalence and mutation-equivalence. Two polygons can be shown to be mutation-equivalent by explicitly calculating a sequence of mutations between the two. Conversely we know that the classical period of the maximally mutable Laurent polynomial corresponding to a polygon is invariant under mutation. Hence given two polytopes with different periods we know there cannot exist a sequence of mutation between the two.

Since $m_{P} = m_{\mathcal{B}}$ we have that the height of the special facet $l_{f} \in \{ 1,\cdots ,m_{\mathcal{B}} \}$. We can translate the top edge $F$ horizontally by $l_{f}$ using a $GL(N)$-transformation, so assume that $-l_{f}< a \leq 0$. It remains to show that for a fixed $l_{f}$ and $a$ that there are only finitely many choices for $b$. Suppose $b \geq a + l_{F}$. Then by minimality the region $T$ contains a point $(x,y)$ with $y \leq -l_{F}$. It is easy to see that if $b$ gets too big then the intersection of $L1$ and $L2$ will bound $T$ so as not to include such a point. Therefore there are only finitely many choices of special facet. Note we only consider $b$ such that $(b,l_{F})$ is primitive and the singularity contributed by $F$ is either a T-singularity or an R-singularity contained in $\mathcal{B}$.

We have successfully written computer code in Sage that efficiently implements the algorithm.

\section{Minimal Fano Polygons with $\mathcal{B} = \{ m_{1} \times \frac{1}{3}(1,1) , m_{2} \times \frac{1}{6}(1,1) \} $. } \label{Section 5}

We apply algorithm 2 to classify all Fano polygons whose only R-singularities are the cyclic quotient singularities $\frac{1}{3}(1,1)$ and $\frac{1}{6}(1,1)$. Set $\mathcal{B}= \{ m_{1} \times \frac{1}{3}(1,1) , m_{2} \times \frac{1}{6}(1,1) \}$ where $m_{1} \in \mathbb{Z}_{\geq 0}$ and $m_{2} \in \mathbb{Z}^{>0}$. $m_{2}$ is non-zero since a classification for Fano polygons with only $\frac{1}{3}(1,1)$ R-singularities has been completed in \cite{MinimalityandMutationEquivalenceofPolygons}. Before we can apply our algorithm, we must find upper bounds on $n, m_{1}$ and $m_{2}$ to prove that we are only required to run the algorithm a finite number of times to complete the classification.

In the $\frac{1}{3}(1,1)$ classifiation of \cite{MinimalityandMutationEquivalenceofPolygons}, we are able to find a bound on the number of R-singularities by substituting the degree contribution $A_{\frac{1}{3}(1,1)}$ into an expression for the anticanonical degree of the corresponding toric variety. However the degree contribution $A_{\frac{1}{6}(1,1)}$ is negative and a similar method does not yield such a bound. We appeal to a combinatorial argument instead.

\begin{lem} \label{5.1}
There exists no minimal Fano polygons $P \subset N_{\mathbb{R}}$, with $m_{P}=3$ and residual basket given by $\mathcal{B}= \{ m\times\frac{1}{6}(1,1) \}$, where $m \geq 3$.
\end{lem}

\begin{proof}
If we prove the case $m=3$ then the result follows for $m>3$.

 Let $P$ be a polygon with $\mathcal{B}= \{ 3\times\frac{1}{6}(1,1) \}$.  By a $GL(N)$-translation we assume that one of the R-singularities is given by $E_{1} = \text{conv} \{ (-1,3), (1,3) \}$. By mutating with respect to any separating T-singularity we assume a second R-singularity is adjacent to $E_{1}$ with one endpoint $(1,3)$, given by an edge $E_{2}$. We aim to mutate with respect to any T-singularities separating $E_{1}$  and the final edge representing an R-singularity, which we denote $E_{3}$, the subtlety being that when we perform such mutations we do not to disrupt $E_{1}$ and $E_{2}$ lying adjacently.

We formalise this construction. Denote the unknown vertex of $E_{2}$ by $(a,b)$. $E_{2}$ must be at height 3. The primitive inner pointing normal of $E_{2}$ is given by 
 \[ n = \Big( \frac{b-3}{g}, \frac{1-a}{g} \Big) \in M \]
 where $g=\text{gcd}(b-3,1-a)$. The height of $E_{2}$ is
 \[ h = -n \cdot (1,3) = \frac{3a-b}{g} . \]
 Setting $h=3$ in the expression as required:
 \[ \frac{3a-b}{g} = 3, \]
 \[ b = 3a -3 \text{ gcd}(b-3,1-a). \]
 
By convexity we require $b<3$. The only remaining integer solution with $a \geq 0$, is given by $(0,-3)$. However this point is not primitive so we can discard it. Hence $a<0$. 

Suppose $E_{3}$ is a vertical edge. By convexity we require $a=-1$ and that $(a,b)$ is a vertex of $E_{3}$. But then $E_{3}$ is of height 1 and $m<3$. Suppose $E_{3}$ is not vertical. Again convexity demands that the second endpoint of $E_{3}$ has first coordinate less than $-1$. Then $\text{height}(E_{3}) > 3$ and we contradict $m_{P} =3$.

Therefore there can be no minimal Fano polygon with residual basket given by $ \mathcal{B}= \{ 3\times\frac{1}{6}(1,1) \}$ with $m_{P} = 3$.
 
\end{proof}

A similar argument shows that if we have a basket $\mathcal{B} = \{ m_{1} \times \frac{1}{3}(1,1) , m_{2} \times \frac{1}{6}(1,1) \}$ as above, then $m_{1} + m_{2} < 3$. Hence we can run the algorithm a finite number of times to get the desired classification.

Examples in this particular classification demonstrate a notion known as shattering introduced by Wormleighton \cite{ReconstructionofOrbifolddelPezzoSurfacesfromHilbertSeries}. Let $C_{1} = \langle u,v \rangle$, $C_{2} = \langle v,w \rangle$ be two cones in $N_{\mathbb{R}}$. Suppose the vectors $v-u$, $w-v$ are parallel. Then we define the hyperplane sum of $C_{1}$ and $C_{2}$ to be given by $C_{1} \ast C_{2} = \langle u,w \rangle$.

\begin{cor}[\cite{ReconstructionofOrbifolddelPezzoSurfacesfromHilbertSeries} Corollary 2.2] \label{5.2}
Let $\sigma_{1} \ast \sigma_{2} \ast \cdots \ast \sigma_{n} = \tau$ be a T-singularity. Then the Riemann-Roch contributions $Q_{\sigma_{i}}$ and the degree contributions $A_{\sigma_{i}}$ satisfy
\[ Q_{\sigma_{1}} + \cdots + Q_{\sigma_{n}} = 0, \]
\[ A_{\sigma_{1}} + \cdots + A_{\sigma_{n}} = A_{\tau} = d = \frac{\text{lattice length}(\tau)}{\text{lattice height}(\tau)}. \]
\end{cor}

Informally this notion can be observed by considering a T-cone at height 3. Let $C = \text{cone}\big\{ (-2,3), (1,3) \big\}$. By adding an additional ray given by primitive generating vector $(-1,3)$ we decompose $C$ into two sub-cones $C_{1}$ and $C_{2}$ representing a $\frac{1}{3}(1,1)$ and a $\frac{1}{6}$(1,1) R-singularity respectively. By Corollary \ref{5.2}
\[ Q_{\frac{1}{3}(1,1)} + Q_{\frac{1}{6}(1,1)} = 0 \]
\[ A_{\frac{1}{3}(1,1)} + A_{\frac{1}{6}(1,1)} = 1. \]
Knowing $A_{\frac{1}{3}(1,1)} = \frac{5}{3}$ and $Q_{\frac{1}{3}(1,1)} = -\frac{t}{3(1-t^{3})}$, we easily derive
\[ A_{\frac{1}{6}(1,1)} = -\frac{2}{3} \]
\[ Q_{\frac{1}{6}(1,1)} = \frac{t}{3(1-t^{3})}. \]

By Proposition \ref{2.13}, Lemma \ref{5.1} and the calculated value of $A_{\frac{1}{6}(1,1)}$, we calculate $(-K_{X_{P}})^{2} = 12 -n - \frac{5}{3}m_{1} + \frac{2}{3}m_{2}$. Since we are interested in Fano polygons, $(-K_{X_{P}})^{2} > 0$, so $n\leq 13$. We only run the algorithm a finite number of times to get the desired classification.

We give the table of results for the classification of polygons with singularity content of the form $\Big(n, \{ m_{1} \times \frac{1}{3}(1,1), m_{2} \times \frac{1}{6}(1,1) \} \Big)$ with $m_{2} \neq 0$, up to mutation-equivalence. All the polytopes listed arose in case (i). Any found in case (ii) are mutation equivalent to one of the polytopes below.

\begin{center}
\def\arraystretch{1.25}
\begin{tabular}{| c | c | c | c | c | c |}
\hline
\# & Vertices of Polygon $P$ & $n$ & $m_{1}$ & $m_{2}$ & $(-K_{X})^{2}$ \\
\hline
\hline
1.1 & (-1,3), (1,3), (0,-1) & 2 & 0 & 1 & $\frac{32}{3}$ \\
\hline
1.2 & (-1,3), (1,3), (1,2), (0,-1) & 3 & 0 & 1 & $\frac{29}{3}$ \\
\hline
1.3 & (-1,3), (1,3), (1,1), (0,-1) & 4 & 0 & 1 & $\frac{26}{3}$ \\
\hline
1.4 & (-1,3), (1,3), (1,0), (0,-1) & 5 & 0 & 1 & $\frac{23}{3}$ \\
\hline
1.5 & (-1,3), (1,3), (1,2), (0,-1), (-1,0) & 6 & 0 & 1 & $\frac{20}{3}$ \\
\hline
1.6 & (-1,3), (1,3), (1,2), (0,-1), (-1,-1) & 7 & 0 & 1 & $\frac{17}{3}$ \\
\hline
1.7 & (-1,3), (1,3), (1,0), (0,-1), (-1,0) & 8 & 0 & 1 & $\frac{14}{3}$ \\
\hline
1.8 & (-1,3), (1,3), (1,0), (-1,-1) & 8 & 0 & 1 & $\frac{14}{3}$ \\
\hline
1.9 & (-1,3), (1,3), (1,0), (0,-1), (-1,-1) & 9 & 0 & 1 & $\frac{11}{3}$ \\
\hline
1.10 & (-1,3), (1,3), (1,2), (-1,-4) & 10 & 0 & 1 & $\frac{8}{3}$ \\
\hline
1.11 & (-1,3), (1,3), (1,-1), (-1,-3) & 11 & 0 & 1 & $\frac{5}{3}$ \\
\hline
1.12 & (-1,3), (1,3), (5,-1), (-5,-1) & 12 & 0 & 1 & $\frac{2}{3}$ \\
\hline
1.13 & (-1,1), (1,1), (5,-1), (-5,-1) & 12 & 0 & 2 & $\frac{4}{3}$ \\
\hline
1.14 & (-1,3), (1,3), (1,-1), (-1,-2) & 9 & 1 & 1 & $2$ \\
\hline
\end{tabular}
\end{center}

These polygons are illustrated in figure 1.

Recall that the maximally mutable Laurent polynomial of a polygon $P$ is a polynomial $f$ such that $\text{Newt}(f) = P$, and that the mutations of $f$ remain Laurent polynomials. We show polygons 1.7 and 1.8 are not mutation-equivalent by calculation of the periods of the associated maximally mutable Laurent polynomials. These periods are mutation invariants by \cite[Lemma 1]{MinkowskiPolynomialsandMutations}.The maximally mutable Laurent polynomials of 1.7 and 1.8 are given respectively by \[f = x y^{3} + 3xy^{2} + a y^{3} + 3xy + by^{2} + x^{-1}y^{3} + x + cy + 3x^{-1} y^{2} + 3x^{-1}y + y^{-1} + x^{-1}, \]  \[g = x y^{3} + 3xy^{2} + d y^{3} + 3xy + ey^{2} + x^{-1}y^{3} + x + fy + 4x^{-1} y^{2} + 6x^{-1}y + 4x^{-1} + x^{-1}y^{-1}. \] Calculating the corresponding periods of $f$ and $g$ we obtain:
\[ \pi_{f} = 1 + (2a +2)x^{2} + (3b + 36)x^{3} + (6a^{2} + 24a + 4c + 186)x^{4} \]  \[ + (20ab + 360a + 60b + 760)x^{5} + \cdots,  \]
\[ \pi_{g} = 1 + 14x^{2} + 6ax^{3} + 546x^{4} + (420a + 30b)x^{5} + \cdots . \]
It is easy to see that this periods are not equal and hence the polygons cannot be mutation-equivalent.

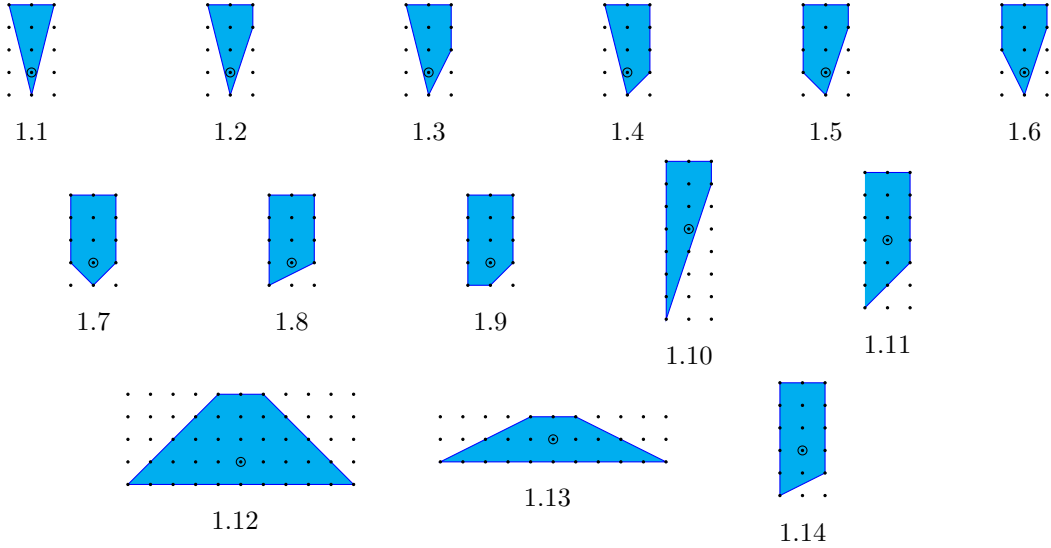
\begin{figure}
\centering
\begin{subfigure}{0.1\textwidth}
\centering
\begin{tikzpicture}[scale=0.3, transform shape]
\begin{scope}
\clip (-1.3,-1.3) rectangle (1.3cm,3.3cm); 
\filldraw[fill=cyan, draw=blue] (-1,3) -- (1,3) -- (0,-1) -- (-1,3); 
\foreach \x in {-7,-6,...,7}{                           
    \foreach \y in {-7,-6,...,7}{                       
    \node[draw,shape = circle,inner sep=1pt,fill] at (\x,\y) {}; 
    }
}
 \node[draw,shape = circle,inner sep=4pt] at (0,0) {}; 
\end{scope}
\end{tikzpicture}
\caption*{1.1}
\end{subfigure}\hspace{1cm}
\begin{subfigure}{0.1\textwidth}
\centering
\begin{tikzpicture}[scale=0.3, transform shape]
\begin{scope}
\clip (-1.3,-1.3) rectangle (1.3cm,3.3cm); 
\filldraw[fill=cyan, draw=blue] (-1,3) -- (1,3) -- (1,2) -- (0,-1) -- (-1,3); 
\foreach \x in {-7,-6,...,7}{                           
    \foreach \y in {-7,-6,...,7}{                       
    \node[draw,shape = circle,inner sep=1pt,fill] at (\x,\y) {}; 
    }
}
 \node[draw,shape = circle,inner sep=4pt] at (0,0) {}; 
\end{scope}
\end{tikzpicture}
\caption*{1.2}
\end{subfigure}\hspace{1cm}
\begin{subfigure}{0.1\textwidth}
\centering
\begin{tikzpicture}[scale=0.3, transform shape]
\begin{scope}
\clip (-1.3,-1.3) rectangle (1.3cm,3.3cm); 
\filldraw[fill=cyan, draw=blue] (-1,3) -- (1,3) -- (1,1) -- (0,-1) -- (-1,3); 
\foreach \x in {-7,-6,...,7}{                           
    \foreach \y in {-7,-6,...,7}{                       
    \node[draw,shape = circle,inner sep=1pt,fill] at (\x,\y) {}; 
    }
}
 \node[draw,shape = circle,inner sep=4pt] at (0,0) {}; 
\end{scope}
\end{tikzpicture}
\caption*{1.3}
\end{subfigure}\hspace{1cm}
\begin{subfigure}{0.1\textwidth}
\centering
\begin{tikzpicture}[scale=0.3, transform shape]
\begin{scope}
\clip (-1.3,-1.3) rectangle (1.3cm,3.3cm); 
\filldraw[fill=cyan, draw=blue] (-1,3) -- (1,3) -- (1,0) -- (0,-1) -- (-1,3); 
\foreach \x in {-7,-6,...,7}{                           
    \foreach \y in {-7,-6,...,7}{                       
    \node[draw,shape = circle,inner sep=1pt,fill] at (\x,\y) {}; 
    }
}
 \node[draw,shape = circle,inner sep=4pt] at (0,0) {}; 
\end{scope}
\end{tikzpicture}
\caption*{1.4}
\end{subfigure}\hspace{1cm}
\begin{subfigure}{0.1\textwidth}
\centering
\begin{tikzpicture}[scale=0.3, transform shape]
\begin{scope}
\clip (-1.3,-1.3) rectangle (1.3cm,3.3cm); 
\filldraw[fill=cyan, draw=blue] (-1,3) -- (1,3) -- (1,2) -- (0,-1) -- (-1,0) -- (-1,3); 
\foreach \x in {-7,-6,...,7}{                           
    \foreach \y in {-7,-6,...,7}{                       
    \node[draw,shape = circle,inner sep=1pt,fill] at (\x,\y) {}; 
    }
}
 \node[draw,shape = circle,inner sep=4pt] at (0,0) {}; 
\end{scope}
\end{tikzpicture}
\caption*{1.5}
\end{subfigure}\hspace{1cm}
\begin{subfigure}{0.1\textwidth}
\centering
\begin{tikzpicture}[scale=0.3, transform shape]
\begin{scope}
\clip (-1.3,-1.3) rectangle (1.3cm,3.3cm); 
\filldraw[fill=cyan, draw=blue] (-1,3) -- (1,3) -- (1,2) -- (0,-1) -- (-1,1) -- (-1,3); 
\foreach \x in {-7,-6,...,7}{                           
    \foreach \y in {-7,-6,...,7}{                       
    \node[draw,shape = circle,inner sep=1pt,fill] at (\x,\y) {}; 
    }
}
 \node[draw,shape = circle,inner sep=4pt] at (0,0) {}; 
\end{scope}
\end{tikzpicture}
\caption*{1.6}
\end{subfigure}\hspace{1cm}

\begin{subfigure}{0.1\textwidth}
\centering
\begin{tikzpicture}[scale=0.3, transform shape]
\begin{scope}
\clip (-1.3,-1.3) rectangle (1.3cm,3.3cm); 
\filldraw[fill=cyan, draw=blue] (-1,3) -- (1,3) -- (1,0) -- (0,-1) -- (-1,0) -- (-1,3); 
\foreach \x in {-7,-6,...,7}{                           
    \foreach \y in {-7,-6,...,7}{                       
    \node[draw,shape = circle,inner sep=1pt,fill] at (\x,\y) {}; 
    }
}
 \node[draw,shape = circle,inner sep=4pt] at (0,0) {}; 
\end{scope}
\end{tikzpicture}
\caption*{1.7}
\end{subfigure}\hspace{1cm}
\begin{subfigure}{0.1\textwidth}
\centering
\begin{tikzpicture}[scale=0.3, transform shape]
\begin{scope}
\clip (-1.3,-1.3) rectangle (1.3cm,3.3cm); 
\filldraw[fill=cyan, draw=blue] (-1,3) -- (1,3) -- (1,0) -- (-1,-1) -- (-1,3); 
\foreach \x in {-7,-6,...,7}{                           
    \foreach \y in {-7,-6,...,7}{                       
    \node[draw,shape = circle,inner sep=1pt,fill] at (\x,\y) {}; 
    }
}
 \node[draw,shape = circle,inner sep=4pt] at (0,0) {}; 
\end{scope}
\end{tikzpicture}
\caption*{1.8}
\end{subfigure}\hspace{1cm}
\begin{subfigure}{0.1\textwidth}
\centering
\begin{tikzpicture}[scale=0.3, transform shape]
\begin{scope}
\clip (-1.3,-1.3) rectangle (1.3cm,3.3cm); 
\filldraw[fill=cyan, draw=blue] (-1,3) -- (1,3) -- (1,0) -- (0,-1) -- (-1,-1) -- (-1,3); 
\foreach \x in {-7,-6,...,7}{                           
    \foreach \y in {-7,-6,...,7}{                       
    \node[draw,shape = circle,inner sep=1pt,fill] at (\x,\y) {}; 
    }
}
 \node[draw,shape = circle,inner sep=4pt] at (0,0) {}; 
\end{scope}
\end{tikzpicture}
\caption*{1.9}
\end{subfigure}\hspace{1cm}
\begin{subfigure}{0.1\textwidth}
\centering
\begin{tikzpicture}[scale=0.3, transform shape]
\begin{scope}
\clip (-1.3,-4.3) rectangle (1.3cm,3.3cm); 
\filldraw[fill=cyan, draw=blue] (-1,3) -- (1,3) -- (1,2) -- (-1,-4) -- (-1,3); 
\foreach \x in {-7,-6,...,7}{                           
    \foreach \y in {-7,-6,...,7}{                       
    \node[draw,shape = circle,inner sep=1pt,fill] at (\x,\y) {}; 
    }
}
 \node[draw,shape = circle,inner sep=4pt] at (0,0) {}; 
\end{scope}
\end{tikzpicture}
\caption*{1.10}
\end{subfigure}\hspace{1cm}
\begin{subfigure}{0.1\textwidth}
\centering
\begin{tikzpicture}[scale=0.3, transform shape]
\begin{scope}
\clip (-1.3,-3.3) rectangle (1.3cm,3.3cm); 
\filldraw[fill=cyan, draw=blue] (-1,3) -- (1,3) -- (1,-1) -- (-1,-3); 
\foreach \x in {-7,-6,...,7}{                           
    \foreach \y in {-7,-6,...,7}{                       
    \node[draw,shape = circle,inner sep=1pt,fill] at (\x,\y) {}; 
    }
}
 \node[draw,shape = circle,inner sep=4pt] at (0,0) {}; 
\end{scope}
\end{tikzpicture}
\caption*{1.11}
\end{subfigure}\hspace{1cm}

\begin{subfigure}{0.2\textwidth}
\centering
\begin{tikzpicture}[scale=0.3, transform shape]
\begin{scope}
\clip (-5.3,-1.3) rectangle (5.3cm,3.3cm); 
\filldraw[fill=cyan, draw=blue] (-1,3) -- (1,3) -- (5,-1) -- (-5,-1) -- (-1,3); 
\foreach \x in {-7,-6,...,7}{                           
    \foreach \y in {-7,-6,...,7}{                       
    \node[draw,shape = circle,inner sep=1pt,fill] at (\x,\y) {}; 
    }
}
 \node[draw,shape = circle,inner sep=4pt] at (0,0) {}; 
\end{scope}
\end{tikzpicture}
\captionsetup{justification=centering}
\caption*{1.12}
\end{subfigure}\hspace{1cm}
\begin{subfigure}{0.2\textwidth}
\centering
\begin{tikzpicture}[scale=0.3, transform shape]
\begin{scope}
\clip (-5.3,-1.3) rectangle (5.3cm,1.3cm); 
\filldraw[fill=cyan, draw=blue] (-1,1) -- (1,1) -- (5,-1) -- (-5,-1) -- (-1,1); 
\foreach \x in {-7,-6,...,7}{                           
    \foreach \y in {-7,-6,...,7}{                       
    \node[draw,shape = circle,inner sep=1pt,fill] at (\x,\y) {}; 
    }
}
 \node[draw,shape = circle,inner sep=4pt] at (0,0) {}; 
\end{scope}
\end{tikzpicture}
\caption*{1.13}
\end{subfigure}\hspace{1cm}
\begin{subfigure}{0.1\textwidth}
\centering
\begin{tikzpicture}[scale=0.3, transform shape]
\begin{scope}
\clip (-1.3,-2.3) rectangle (1.3cm,3.3cm); 
\filldraw[fill=cyan, draw=blue] (-1,3) -- (1,3) -- (1,-1) -- (-1,-2) -- (-1,3); 
\foreach \x in {-7,-6,...,7}{                           
    \foreach \y in {-7,-6,...,7}{                       
    \node[draw,shape = circle,inner sep=1pt,fill] at (\x,\y) {}; 
    }
}
 \node[draw,shape = circle,inner sep=4pt] at (0,0) {}; 
\end{scope}
\end{tikzpicture}
\caption*{1.14}
\end{subfigure}\hspace{1cm}

\caption*{Figure 1: Minimal Representatives of Mutation-Equivalence Classes of Fano Polygons with Singularity Content $\Big(n, \{ m_{1} \times \frac{1}{3}(1,1), m_{2} \times \frac{1}{6}(1,1) \} \Big)$ where $m_{1} \geq 0, m_{2} > 0$.}
\end{figure}

\section{Minimal Fano Polygons with $ \mathcal{B} = \{ m \times \frac{1}{5}(1,1) \} $. } \label{Section 6}

We find all Fano polygons with singularity content of the form $\Big(n, \{ m \times \frac{1}{5}(1,1) \} \Big)$ with $m > 0$. Similarly to section \ref{Section 5}, we find bounds on $n$ and $m$ to ensure we find a complete classification.

\begin{lem}
There exist no minimal Fano polygons $P \subset N_{\mathbb{R}}$, with $m_{P} = 5$ and residual basket given by $\mathcal{B} = \{ m \times \frac{1}{5}(1,1)\}$, where $m \geq 3$.
\end{lem}

\begin{proof}
Similarly to the proof of Lemma \ref{5.1}, we assume the existence of a Fano polygon $P$ with three $\frac{1}{5}(1,1)$ singularities, and perform a combination of $GL(N)$-translations and mutations so that one of the R-singularities is represented by the edge $E_{1} = \text{conv} \{(-3,5),(-2,5) \}$, and another by $E_{2} = \text{conv}\{(-2,5), (a,b)\}$, where $(a,b)\neq (-3,5)$. We show that we can always mutate $P$ so that the third R-singularity is represented by $E_{3}= \text{conv}\{(-3,5),(c,d)\}$, where $(c,d) \neq (-3,5)$, without disrupting the original two singularities sitting adjacently.

We study the possible T-cones that, when mutated with respect to, would separate the adjacent R-singularities. We calculate the line of points from (-2,5) that give an edge at height 5. Unlike Lemma \ref{5.1}, since we are only interested in $\frac{1}{5}(1,1)$ singularities, we can assume that the inner pointing normal $n = (b-5, -2-a)$ is primitive. This line provides a bound in which $(c,d)$ lies by convexity. Convexity also determines that $d \leq 5$. Furthermore since $P$ is Fano, the origin $(0,0)$ must lie in the interior so we further bound the region $(c,d)$ lies in. Finally since we are only interested in the case where the prospective T-singularity would disrupt the adjacent R-cones when mutated with respect to we obtain a final bound on the region in which $(c,d)$ can lie. It is then possible to exhaustively check that none of the primitive lattice points in this region give the second vertex of a T-cone.

Hence we assume that the three R-cones lie adjacently. Calculating the points $(c,d)$ so that $E_{3}$ is height 5 and comparing with the possible choices of $(a,b)$, we see that there are so choice of $(a,b)$ and $(c,d)$ that maintain convexity.

Therefore there can be no minimal Fano polygon with residual basket given by $\mathcal{B} = \{ 3 \times \frac{1}{5}(1,1) \}$ with $m_{P} = 5$.
\end{proof}

We know from Example \ref{2.14} that the anticanonical degree of the toric variety corresponding to a Fano polygon with only R-singularities of type $\frac{1}{5}(1,1)$ is given by $(-K_{X_{P}})^{2} = 12 -n - \frac{1}{5}m$. Therefore $n < 12$. We apply algorithm 2 finitely many times to complete the classification.

We give the table of results for the classification of polygons with singuarity content of the form $\Big(n, \{ m \times \frac{1}{5}(1,1) \} \Big)$ with $m > 0$. All the polytopes were found in case (i). None arose in case (ii).

\begin{center}
\def\arraystretch{1.25}
\begin{tabular}{| c | c | c | c | c |}
\hline
\# & Vertices of Polygon $P$ & $n$ & $m$ & $(-K_{X})^{2}$ \\
\hline
\hline
2.1 & (-3,5), (-2,5), (1,-2)  & 2 & 1 & $\frac{49}{5}$ \\
\hline
2.2 &(-3,5), (-2,5), (-1,3), (1,-2)  & 3 & 1 & $\frac{44}{5}$ \\
\hline
2.3 &(-3,5), (-2,5), (-1,3), (1,-2), (-2,3)  & 4 & 1 & $\frac{39}{5}$ \\
\hline
2.4 & (-3,5), (-2,5), (-1,3), (1,-2), (-1,1) & 5 & 1 & $\frac{34}{5}$ \\
\hline
2.5 &(-3,5), (-2,5), (0,1), (1,-2), (-1,1)  & 6 & 1 & $\frac{29}{5}$ \\
\hline
2.6 &(-3,5), (-2,5), (0,1), (1,-2), (0,-1)  & 7 & 1 & $\frac{24}{5}$ \\
\hline
2.7 & (-3,5), (-2,5), (1,-1), (0,-1) & 7 & 1 & $\frac{24}{5}$ \\
\hline
2.8 &(-3,5), (-2,5), (1,-1), (1,-2), (0,-1) & 8 & 1 & $\frac{19}{5}$ \\
\hline
2.9 &(-3,5), (-2,5), (1,-1), (1,-3) & 9 & 1 & $\frac{14}{5}$ \\
\hline
2.10 &(-3,5), (-2,5), (2,-3), (2,-5) & 10 & 1 & $\frac{9}{5}$ \\
\hline
2.11 &(-3,5), (-2,5), (4,-1), (-3,-1) & 11 & 1 & $\frac{4}{5}$ \\
\hline
2.12 &(-3,5), (-2,5), (3,-5), (2,-5) & 10 & 2 & $\frac{8}{5}$ \\
\hline
\end{tabular}
\end{center}

These polygons are illustrated in figure 2.

Similarly to section \ref{Section 5} we see that polygons 2.6 and 2.7 are not mutation equivalent by looking at the periods $\pi_{f}, \pi_{g}$ of their respective maximally mutable Laurent polynomials $f$ and $g$:
\[ \pi_{f}  = 1 + 12x^{2} + 6ax^{3} + 396x^{4} + (360a + 30b) x^{5} + \cdots, \] 
\[ \pi_{g} = 1 + (2c+12)x^{2} + (6c + 3d +90)x^{3} + (6c^{2}+24d +144c +636)x^{4} \] 
\[ + (20cd + 60c^{2} +390d +1260c +6900)x^{5} + \cdots . \]
It follows that polygons 2.6 and 2.7 are not mutation-equivalent.

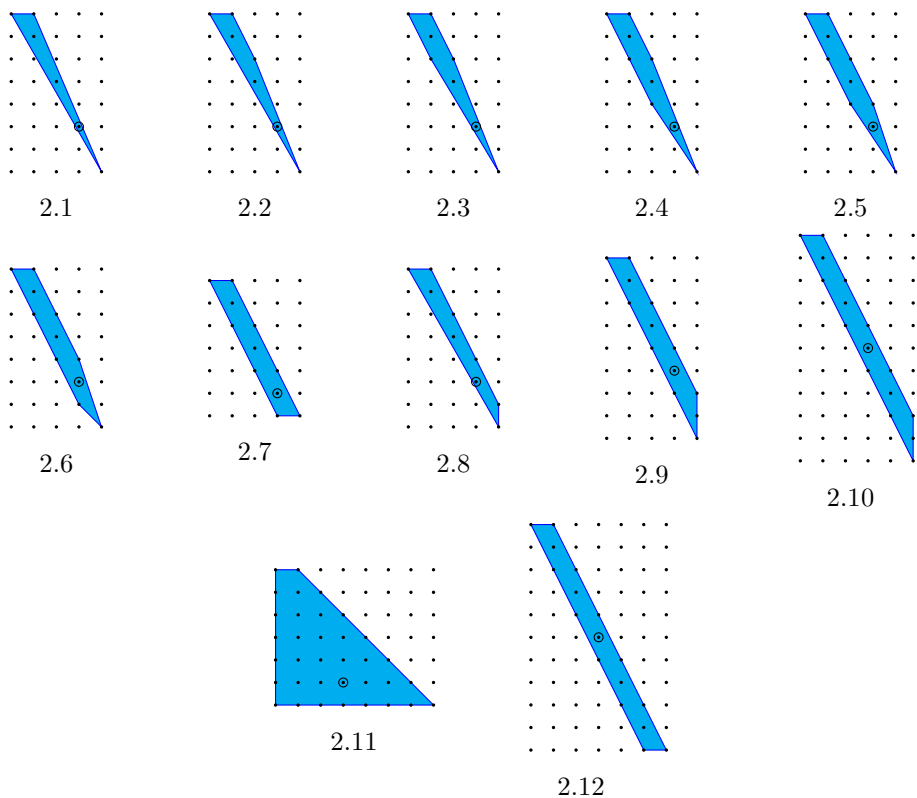
\begin{figure}
\centering
\begin{subfigure}{0.1\textwidth}
\centering
\begin{tikzpicture}[scale=0.3, transform shape]
\begin{scope}
\clip (-3.3,-2.3) rectangle (1.3cm,5.3cm); 
\filldraw[fill=cyan, draw=blue] (-3,5) -- (-2,5) -- (1,-2) -- (-3,5); 
\foreach \x in {-7,-6,...,7}{                           
    \foreach \y in {-7,-6,...,7}{                       
    \node[draw,shape = circle,inner sep=1pt,fill] at (\x,\y) {}; 
    }
}
 \node[draw,shape = circle,inner sep=4pt] at (0,0) {}; 
\end{scope}
\end{tikzpicture}
\caption*{2.1}
\end{subfigure}\hspace{1cm}
\begin{subfigure}{0.1\textwidth}
\centering
\begin{tikzpicture}[scale=0.3, transform shape]
\begin{scope}
\clip (-3.3,-2.3) rectangle (1.3cm,5.3cm); 
\filldraw[fill=cyan, draw=blue] (-3,5) -- (-2,5) -- (-1,3) -- (1,-2) -- (-3,5); 
\foreach \x in {-7,-6,...,7}{                           
    \foreach \y in {-7,-6,...,7}{                       
    \node[draw,shape = circle,inner sep=1pt,fill] at (\x,\y) {}; 
    }
}
 \node[draw,shape = circle,inner sep=4pt] at (0,0) {}; 
\end{scope}
\end{tikzpicture}
\caption*{2.2}
\end{subfigure}\hspace{1cm}
\begin{subfigure}{0.1\textwidth}
\centering
\begin{tikzpicture}[scale=0.3, transform shape]
\begin{scope}
\clip (-3.3,-2.3) rectangle (1.3cm,5.3cm); 
\filldraw[fill=cyan, draw=blue] (-3,5) -- (-2,5) -- (-1,3) -- (1,-2) -- (-2,3) -- (-3,5); 
\foreach \x in {-7,-6,...,7}{                           
    \foreach \y in {-7,-6,...,7}{                       
    \node[draw,shape = circle,inner sep=1pt,fill] at (\x,\y) {}; 
    }
}
 \node[draw,shape = circle,inner sep=4pt] at (0,0) {}; 
\end{scope}
\end{tikzpicture}
\caption*{2.3}
\end{subfigure}\hspace{1cm}
\begin{subfigure}{0.1\textwidth}
\centering
\begin{tikzpicture}[scale=0.3, transform shape]
\begin{scope}
\clip (-3.3,-2.3) rectangle (1.3cm,5.3cm); 
\filldraw[fill=cyan, draw=blue] (-3,5) -- (-2,5) -- (-1,3) -- (1,-2) -- (-1,1) -- (-3,5); 
\foreach \x in {-7,-6,...,7}{                           
    \foreach \y in {-7,-6,...,7}{                       
    \node[draw,shape = circle,inner sep=1pt,fill] at (\x,\y) {}; 
    }
}
 \node[draw,shape = circle,inner sep=4pt] at (0,0) {}; 
\end{scope}
\end{tikzpicture}
\caption*{2.4}
\end{subfigure}\hspace{1cm}
\begin{subfigure}{0.1\textwidth}
\centering
\begin{tikzpicture}[scale=0.3, transform shape]
\begin{scope}
\clip (-3.3,-2.3) rectangle (1.3cm,5.3cm); 
\filldraw[fill=cyan, draw=blue] (-3,5) -- (-2,5) -- (0,1) -- (1,-2) -- (-1,1) -- (-3,5); 
\foreach \x in {-7,-6,...,7}{                           
    \foreach \y in {-7,-6,...,7}{                       
    \node[draw,shape = circle,inner sep=1pt,fill] at (\x,\y) {}; 
    }
}
 \node[draw,shape = circle,inner sep=4pt] at (0,0) {}; 
\end{scope}
\end{tikzpicture}
\caption*{2.5}
\end{subfigure}\hspace{1cm}

\begin{subfigure}{0.1\textwidth}
\centering
\begin{tikzpicture}[scale=0.3, transform shape]
\begin{scope}
\clip (-3.3,-2.3) rectangle (1.3cm,5.3cm); 
\filldraw[fill=cyan, draw=blue] (-3,5) -- (-2,5) -- (0,1) -- (1,-2) -- (0,-1) -- (-3,5); 
\foreach \x in {-7,-6,...,7}{                           
    \foreach \y in {-7,-6,...,7}{                       
    \node[draw,shape = circle,inner sep=1pt,fill] at (\x,\y) {}; 
    }
}
 \node[draw,shape = circle,inner sep=4pt] at (0,0) {}; 
\end{scope}
\end{tikzpicture}
\caption*{2.6}
\end{subfigure}\hspace{1cm}
\begin{subfigure}{0.1\textwidth}
\centering
\begin{tikzpicture}[scale=0.3, transform shape]
\begin{scope}
\clip (-3.3,-1.3) rectangle (1.3cm,5.3cm); 
\filldraw[fill=cyan, draw=blue] (-3,5) -- (-2,5) -- (1,-1) -- (0,-1) -- (-3,5); 
\foreach \x in {-7,-6,...,7}{                           
    \foreach \y in {-7,-6,...,7}{                       
    \node[draw,shape = circle,inner sep=1pt,fill] at (\x,\y) {}; 
    }
}
 \node[draw,shape = circle,inner sep=4pt] at (0,0) {}; 
\end{scope}
\end{tikzpicture}
\caption*{2.7}
\end{subfigure}\hspace{1cm}
\begin{subfigure}{0.1\textwidth}
\centering
\begin{tikzpicture}[scale=0.3, transform shape]
\begin{scope}
\clip (-3.3,-2.3) rectangle (1.3cm,5.3cm); 
\filldraw[fill=cyan, draw=blue] (-3,5) -- (-2,5) -- (1,-1) -- (1,-2) -- (-3,5); 
\foreach \x in {-7,-6,...,7}{                           
    \foreach \y in {-7,-6,...,7}{                       
    \node[draw,shape = circle,inner sep=1pt,fill] at (\x,\y) {}; 
    }
}
 \node[draw,shape = circle,inner sep=4pt] at (0,0) {}; 
\end{scope}
\end{tikzpicture}
\caption*{2.8}
\end{subfigure}\hspace{1cm}
\begin{subfigure}{0.1\textwidth}
\centering
\begin{tikzpicture}[scale=0.3, transform shape]
\begin{scope}
\clip (-3.3,-3.3) rectangle (1.3cm,5.3cm); 
\filldraw[fill=cyan, draw=blue] (-3,5) -- (-2,5) -- (1,-1) -- (1,-3) -- (-3,5); 
\foreach \x in {-7,-6,...,7}{                           
    \foreach \y in {-7,-6,...,7}{                       
    \node[draw,shape = circle,inner sep=1pt,fill] at (\x,\y) {}; 
    }
}
 \node[draw,shape = circle,inner sep=4pt] at (0,0) {}; 
\end{scope}
\end{tikzpicture}
\caption*{2.9}
\end{subfigure}\hspace{1cm}
\begin{subfigure}{0.1\textwidth}
\centering
\begin{tikzpicture}[scale=0.3, transform shape]
\begin{scope}
\clip (-3.3,-5.3) rectangle (2.3cm,5.3cm); 
\filldraw[fill=cyan, draw=blue] (-3,5) -- (-2,5) -- (2,-3) -- (2,-5) -- (-3,5); 
\foreach \x in {-7,-6,...,7}{                           
    \foreach \y in {-7,-6,...,7}{                       
    \node[draw,shape = circle,inner sep=1pt,fill] at (\x,\y) {}; 
    }
}
 \node[draw,shape = circle,inner sep=4pt] at (0,0) {}; 
\end{scope}
\end{tikzpicture}
\caption*{2.10}
\end{subfigure}\hspace{1cm}

\begin{subfigure}{0.15\textwidth}
\centering
\begin{tikzpicture}[scale=0.3, transform shape]
\begin{scope}
\clip (-3.3,-1.3) rectangle (4.3cm,5.3cm); 
\filldraw[fill=cyan, draw=blue] (-3,5) -- (-2,5) -- (4,-1) -- (-3,-1) -- (-3,5); 
\foreach \x in {-7,-6,...,7}{                           
    \foreach \y in {-7,-6,...,7}{                       
    \node[draw,shape = circle,inner sep=1pt,fill] at (\x,\y) {}; 
    }
}
 \node[draw,shape = circle,inner sep=4pt] at (0,0) {}; 
\end{scope}
\end{tikzpicture}
\caption*{2.11}
\end{subfigure}\hspace{1cm}
\begin{subfigure}{0.1\textwidth}
\centering
\begin{tikzpicture}[scale=0.3, transform shape]
\begin{scope}
\clip (-3.3,-5.3) rectangle (3.3cm,5.3cm); 
\filldraw[fill=cyan, draw=blue] (-3,5) -- (-2,5) -- (3,-5) -- (2,-5) -- (-3,5); 
\foreach \x in {-7,-6,...,7}{                           
    \foreach \y in {-7,-6,...,7}{                       
    \node[draw,shape = circle,inner sep=1pt,fill] at (\x,\y) {}; 
    }
}
 \node[draw,shape = circle,inner sep=4pt] at (0,0) {}; 
\end{scope}
\end{tikzpicture}
\caption*{2.12}
\end{subfigure}\hspace{1cm}

\caption*{Figure 2: Minimal Representatives of Mutation-Equivalence Classes of Fano Polygons with Singularity Content $\Big(n, \{ m \times \frac{1}{5}(1,1) \} \Big)$ where $m > 0$.}
\end{figure}

\section*{Acknowledgements}

We thank Alexander Kasprzyk for his guidance and many useful conversations. Much of the paper was written in during a visit of EK to Nottingham supported by EPRSC Fellowship EP/NO22513/1.

School of Mathematical Sciences, University of Nottingham, Nottingham, NG7 2RD, UK. \textit{E-mail address:} danielcavey27@gmail.com

Mathematics Institute, University of Nottingham, Coventry, CV4 7AL, UK. \textit{E-mail address:} E.Kutas@warwick.ac.uk
\end{document}